\documentclass[10pt]{amsart}
\usepackage{amsmath}
\usepackage{amssymb}%
\usepackage{amsthm}
\usepackage[numbers, square]{natbib}

\theoremstyle{plain}
\newtheorem{theorem}{Theorem}[]
\newtheorem{lemma}{Lemma}[]
\newtheorem{proposition}{Proposition}[]
\theoremstyle{definition}

\theoremstyle{remark}
\newtheorem{remark}{Remark}[]

\newcommand{\eqdistr}{\stackrel{d}{=}}
\renewcommand{\P}{\mathbb P}
\newcommand{\AAA}{\mathcal A}

\newcommand{\E}{\mathbb E}
\newcommand{\R}{\mathbb{R}}
\newcommand{\N}{\mathbb{N}}

\newcommand{\Z}{\mathbb{Z}}

\newcommand{\eps}{\varepsilon}

\newcommand{\MMM}{\mathbb{M}}

\begin{document}

\title{Persistence of competing systems of branching random walks}
\author{Zakhar Kabluchko}
\address{Institute of Stochastics, Ulm University, Helmholtzstr.\ 18, 89069 Ulm, Germany}
\begin{abstract}
We consider a system of independent branching random walks on $\R$  which start off a Poisson point process with intensity of the form $e_{\lambda}(du)=e^{-\lambda u}du$, where $\lambda\in\R$ is chosen in such a way that the overall intensity of particles is preserved. Denote by $\chi$ the cluster distribution and let $\varphi$ be the log-Laplace transform of the intensity of $\chi$. If $\lambda\varphi'(\lambda)>0$, we show that the system is persistent (stable) meaning that the point process formed by the particles in the $n$-th generation converges as $n\to\infty$ to a non-trivial point process $\Pi_{e_{\lambda}}^{\chi}$ with intensity $e_{\lambda }$. If $\lambda\varphi'(\lambda)<0$, then the branching population suffers local extinction meaning that the limiting point process is empty. We characterize (generally, non-stationary) point processes on $\R$ which are cluster-invariant  with respect to the cluster distribution $\chi$ as mixtures of the point processes $\Pi_{ce_{\lambda}}^{\chi}$ over $c>0$ and $\lambda\in K_{\text{st}}$, where $K_{\text{st}}=\{\lambda\in\R: \varphi(\lambda)=0, \lambda\varphi'(\lambda)>0\}$.
\end{abstract}
\maketitle

\section{Introduction}\label{sec:main}




\subsection{Persistence criterion}
We consider a population of particles on $\R^d$ whose stochastic evolution (branching dynamics) is governed by the following rules. The initial positions of the particles form a Poisson point process on $\R^d$ denoted by $\pi_0$. The intensity measure of $\pi_0$, denoted by $\nu$, is  assumed to be finite on bounded sets. After one unit of time each particle is replaced, independently of all other particles, by a random cluster of offsprings whose displacements with respect to the parent particle are governed by a point process $\chi$ (which is allowed to be empty). Note that the distribution of $\chi$ does not depend on the position of the parent particle.     The daughter particles form a point process denoted by $\pi_1$. Then, any of the daughter particles is independently replaced by a random cluster of granddaughter particles forming a point process $\pi_2$, and so on.  We denote the point process formed by the particles of the $n$-th generation by $\pi_n$.
We will be interested in the behavior of $\pi_n$ as $n\to\infty$.

This problem has been much studied especially if the cluster distribution $\chi$ is critical (i.e., the mean number of particles in the cluster $\chi$ is $1$) and $\nu$ is the Lebesgue measure; see~\cite{kallenberg,ivanoff,gorostiza_wakolbinger,wakolbinger_spatial} and~\cite[Chapters 11, 12]{matthes_etal_book}.
In this particular case it has been shown, under a second moment condition on the intensity of the clusters,  that in dimensions $d=1$ and $d=2$ the branching dynamics suffers local extinction. This means that the point process $\pi_n$ converges as $n\to\infty$ to the empty point process. In dimension $d\geq 3$, the branching dynamics is persistent (stable) meaning that $\pi_n$ converges to some non-trivial point process whose intensity is equal to the Lebesgue measure. A general theory of convergence to equilibrium of branching populations on an arbitrary Polish space was developed in~\cite{main_book}; see also~\cite{wakolbinger_book}.

We will be interested what happens if the  Poisson point process $\pi_0$ formed by the initial positions of the particles is not homogeneous, i.e., the measure $\nu$ is not the Lebesgue measure.
We will need the following assumption.  Let $J$ be the intensity of the point process $\chi$, that is $J$ is a measure on $\R^d$ given by $J(B)=\E[\chi(B)]$ for every Borel set $B\subset \R^d$. Assume that the log-Laplace transform of $J$ is finite:
\begin{equation}\label{eq:def_varphi}
\varphi(t):=\log \int_{\R^d} e^{\langle t, u\rangle} J(du)<\infty,\;\;\; t\in\R^d.
\end{equation}
Here, $\langle \cdot,\cdot\rangle$ denotes the scalar product on $\R^d$.

When studying persistence of a branching dynamics it is natural to require that the intensity of the point process $\pi_n$ should remain constant in time.
This leads to a convolution equation
\begin{equation}\label{eq:deny}
J*\nu=\nu.
\end{equation}
The solutions of this equation have been described by~\citet{deny}.
For $\lambda\in\R^d$ let $e_{\lambda}$ be a measure on $\R^d$ whose density with respect to the Lebesgue measure is given by
\begin{equation}\label{eq:def_e_lambda}
e_{\lambda}(du)=e^{-\langle\lambda,u\rangle}du.
\end{equation}
An easy computation shows that any measure of the form $e_{\lambda}$, where $\lambda$ is such that $\varphi(\lambda)=0$, is a solution to~\eqref{eq:deny}. Conversely, under a non-lattice assumption on $J$, \citet{deny} showed that any measure $\nu$ which solves~\eqref{eq:deny} and is finite on bounded sets can be represented as a mixture of the form
\begin{equation}\label{eq:deny_mixture}
\nu(\cdot)=\int_{K} e_{\lambda}(\cdot) \mu(d\lambda),
\end{equation}
where $\mu$ is a finite measure on the set $K:=\{\lambda\in\R^d: \varphi(\lambda)=0\}$.

In the next theorem we give a criterion for persistensy/local extinction of the branching dynamics starting off a Poisson process with intensity measure $e_{\lambda}$, $\lambda\in K$.
Denote by $\nabla$ the gradient operator.
\begin{theorem}\label{theo:main}
Consider a branching dynamics satisfying~\eqref{eq:def_varphi} which starts off a Poisson point process with intensity $e_{\lambda}$, where $\lambda\in\R^d$ is such that $\varphi(\lambda)=0$.
\begin{enumerate}
\item \label{p:main_1}
If $\langle \lambda,  \nabla \varphi(\lambda)\rangle>0$, then the point process $\pi_n$ converges weakly to some point process $\Pi_{e_{\lambda}}^{\chi}$ having intensity $e_{\lambda}$ (the branching dynamics is persistent).
\item \label{p:main_2}
If $\langle \lambda,  \nabla \varphi(\lambda)\rangle<0$, then the point process $\pi_n$ converges weakly to the empty point process (the branching dynamics suffers local extinction).
\end{enumerate}
\end{theorem}

\subsection{The one-dimensional case}
It is instructive to consider the special case $d=1$. To exclude trivialities, we assume that $J$ is not concentrated on $[0,+\infty)$ or $(-\infty, 0]$. This implies that  $\lim_{t\to\pm\infty}\varphi(t)=+\infty$. The function $\varphi$ is strictly convex.

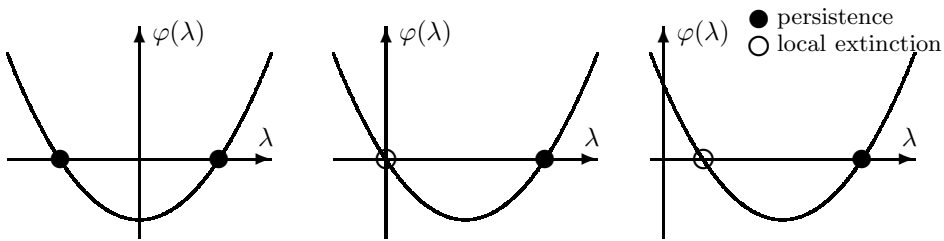
\begin{center}
\begin{picture}(400,80)
\setlength{\unitlength}{1pt}
\thicklines
\put(50,0){\vector(0,1){80}}
\put(0,30){\vector(1,0){100}}
\put(95,35){$\lambda$}
\put(55,75){$\varphi(\lambda)$}
\qbezier(0,70)(50,-56)(100,70)
\put(20,30){\circle*{7}}
\put(80,30){\circle*{7}}
\put(120,0)
{
\begin{picture}(100,100)
\put(20,0){\vector(0,1){80}}
\put(0,30){\vector(1,0){100}}
\put(95,35){$\lambda$}
\put(25,75){$\varphi(\lambda)$}
\qbezier(0,70)(50,-56)(100,70)
\put(20,30){\circle{7}}
\put(80,30){\circle*{7}}
\end{picture}
}
\put(240,0)
{
\begin{picture}(100,100)
\put(5,0){\vector(0,1){80}}
\put(0,30){\vector(1,0){100}}
\put(95,35){$\lambda$}
\put(10,75){$\varphi(\lambda)$}
\qbezier(0,70)(50,-56)(100,70)
\put(20,30){\circle{7}}
\put(80,30){\circle*{7}}
\end{picture}
}
\put(305, 75)
{
\begin{picture}(20,20)
\put(-24,8){\circle*{7}}
\put(-24,-2){\circle{7}}
\put(-17,6){\small persistence}
\put(-17,-4){\small local extinction}
\end{picture}
}
\end{picture}
\end{center}

Assume first that the cluster distribution $\chi$ is subcritical, that is, the mean number of offsprings of any particle is less than $1$. Note that this implies that the progeny of any individual particle dies out with probability $1$; see~\cite{athreya_book}. Let us see what happens to the population as a whole. Since the subcriticality means that $\varphi(0)<0$, the equation $\varphi(\lambda)=0$ has two solutions $\lambda_1<0$ and $\lambda_2>0$ which satisfy $\varphi'(\lambda_1)<0$ and $\varphi'(\lambda_2)>0$. By Theorem~\ref{theo:main}, both measures $e_{\lambda_1}$ and $e_{\lambda_2}$ give rise to a persistent branching dynamics.
In both cases we have a system of branching random walks which become extinct individually but survive as a collective.

Assume now  that the cluster distribution $\chi$ is critical, that is the mean number of offsprings of any particle is $1$. Equivalently, $\varphi(0)=0$. Again, the progeny of any individual particle dies out with probability $1$; see~\cite{athreya_book}.  If $\varphi'(0)=\int_{\R}u J(du)=0$, then the equation $\varphi(\lambda)=0$ has only one solution $\lambda=0$. It is known, see~\cite{kallenberg}, that the corresponding branching dynamics, starting off the Lebesgue measure, suffers local extinction.
Therefore, we assume  that $\int_{\R}u J(du)<0$ which means that the particles have a drift to the left. (The case $\int_{\R}u J(du)>0$ is analogous.)  Then, $\varphi(0)=0$ and $\varphi'(0)<0$. Hence, the equation $\varphi(\lambda)=0$ has two solutions $\lambda_1=0$ and $\lambda_2>0$.
By~\cite{kallenberg}, the first solution $\lambda_1=0$ gives rise to a branching dynamics which suffers local extinction. However, by Theorem~\ref{theo:main}, the second solution $\lambda_2$ corresponds to a persistent branching dynamics. Thus, a critical branching dynamics can survive provided that the particles have a non-zero drift and they have chosen their starting positions in a right way.

Assume finally that the cluster distribution $\chi$ is supercritical, i.e., the mean number of offsprings of any particle is greater than $1$. Equivalently, $\varphi(0)>0$.   There are $3$ possibilities. First, it may happen that the equation $\varphi(\lambda)=0$ has no solutions, which is not interesting to us. Second, it may happen that the equation $\varphi(\lambda)=0$ has two different solutions of the same sign. By Theorem~\ref{theo:main}, the solution which has larger distance  to $0$ gives rise to a persistent branching dynamics, whereas the branching dynamics corresponding to the remaining solution becomes locally extinct. Thus, a supercritical branching dynamics may suffer local extinction, even though the progeny of any particle survives with positive probability and grows exponentially on survival. Finally,  it may happen that the equation $\varphi(\lambda)=0$ has only one solution $\lambda_1$ such that $\varphi'(\lambda_1)=0$. This boundary case is not covered by Theorem~\ref{theo:main}. In Section~\ref{sec:boundary} we will show that the branching dynamics dies out in the boundary case if the underlying branching process is the branching Brownian motion, and conjecture that this holds for general branching random walks. The boundary case has been an object of recent interest due to its connection to the extremal process of the branching Brownian motion; see~\cite{lalley_sellke,brunet_derrida1,brunet_derrida2,arguin_etal3}.



A very simple special case arises if the cluster  $\chi$ consists a.s.\ of exactly $1$ particle (no branching). The persistence of the system is evident, since $\pi_n$ is a Poisson process with the same intensity $e_{\lambda}$ for every $n$. See~\cite{brown_resnick,kabluchko_schlather_dehaan,kabluchko,aizenman_ruzmaikina} for various types of exponential intensity particle systems without branching.

\section{Properties of the competing branching dynamics}
In the one-dimensional case $d=1$ the branching dynamics described in the introduction can be seen as a competition between branching particles.
In the sequel, we will study the properties of this competition always assuming  $d=1$.
\subsection{Notation}
Let $\MMM$ be the space of discrete point configurations (locally finite counting measures) on $\R$. A point process is a random element of $\MMM$.  Given a discrete point configuration $\rho=\sum_{i\in I}\delta_{V_i}\in\MMM$, where the index set $I$ is at most countable, we denote by $\max \rho:=\max_{i\in I} V_i$ the position of the right-most atom (the leader, the most successful particle) of $\rho$ (provided that it exists).  Also, for $u\in\R$ we denote by $T_u\rho=\sum_{i\in I}\delta_{u+V_i}$ the point configuration obtained by shifting every atom $V_i$ of $\rho$ by $u$. Let $\N_0=\{0,1,\ldots\}$ be the set of non-negative integers.

We fix some point process $\chi$ (the distribution of clusters) with intensity $J$ satisfying~\eqref{eq:def_varphi}. Note that the number of particles in $\chi$ is a.s.\ finite. A branching random walk with cluster distribution $\chi$ is a process which starts at time $0$ with a single particle located at the origin and in which every unit of time every particle is independently replaced by a cluster of daughter particles whose displacements with respect to the parent particle are governed by the point process $\chi$. We denote by $\chi_n$ the point process formed by the particles in the $n$-th generation of the branching random walk. By the Hammersley--Kingman--Biggins theorem~\cite{biggins1,biggins}, the maximum $\max \chi_n$ of a supercritical branching random walk $\chi_n$ satisfies
\begin{equation}\label{eq:biggins}
\lim_{n\to\infty}\frac 1n \max \chi_n=\beta_0 \;\;\;\text{a.s. on survival},
\end{equation}
where $\beta_0$ is the largest zero of $I(z)=\sup_{t\in\R} (zt -\varphi(t))$, the Legendre--Fenchel transform of $\varphi$. Note that $I$ is allowed to take the value $+\infty$.

 In the sequel, we denote by $\pi_0=\sum_{i=1}^{\infty}\delta_{U_i}$ a Poisson point process on $\R$ with intensity $e_{\lambda}(du)=e^{-\lambda u}du$. With this notation, the positions of the particles in the $n$-th generation of a branching dynamics starting off $\pi_0$ form a point process
\begin{equation}\label{eq:def_pi_n}
\pi_n=\sum_{i=1}^{\infty}T_{U_i}(\chi_n^{(i)}),
\end{equation}
where $\{\chi_n^{(i)}: n\in\N_0\}$, $i\in\N$, are copies of the branching random walk $\{\chi_n: n\in\N_0\}$, all objects being independent.

\subsection{The leading particle}
Let us look at the progeny of an individual particle in the branching dynamics $\{\pi_n:n\in\N_0\}$. We claim that even in the persistent case, the progeny of any particle with probability $1$ either dies out or drifts to $+\infty$ or $-\infty$.  Indeed, if the cluster distribution $\chi$ is subcritical or critical, then the progeny of any particle dies out a.s.; see~\cite{athreya_book}. If the cluster distribution is supercritical, then the progeny survives with positive probability. However, assuming for concreteness that $\lambda>0$ and recalling that $\varphi(\lambda)=0$, it is easy to see that the limit $\beta_0$ in~\eqref{eq:biggins} is negative. Hence, the progeny of any individual particle either dies out or drifts to $-\infty$ in a uniform way.

These considerations suggest that the leader of a persistent branching dynamics at time $n$, where $n$ is large, is most likely to be an offspring of some particle which had a very unfavorable position at time $0$ (we assume $\lambda>0$ here). The next proposition makes this precise.  Let $i_n$ denote the (random) number of the particle which is the ancestor at time $0$ of the leader of the system at time $n$. That is, $i_n$ is such that
$\max \pi_n= U_{i_n}+\max \chi_n^{(i_n)}$.
\begin{proposition}\label{prop:leader_stable}
Assume that $\varphi(\lambda)=0$, $\lambda>0$, $\varphi'(\lambda)>0$. Then, we have
$
\lim_{n\to\infty} \frac 1n U_{i_n}=-\varphi'(\lambda)<0
$
a.s.
\end{proposition}
In the supercritical or critical case, let $\beta_0$ be the largest zero of the Legendre--Fenchel transform $I$. In the subcritical case, we set $\beta_0=-\infty$. In the next proposition we show that the maximum of a supercritical branching dynamics $\pi_n$ suffering local extinction converges to $-\infty$ with the same asymptotic linear speed $\beta_0$ as the most successful offspring of a single particle.
\begin{proposition}\label{prop:unstab_speed}
Assume that $\varphi(\lambda)=0$, $\lambda>0$, $\varphi'(\lambda)<0$. Then, we have
$\lim_{n\to\infty} \frac 1n\max \pi_n =\beta_0$ a.s.
\end{proposition}
Let us mention also the following proposition which will be obtained as a simple corollary of Theorem~\ref{theo:main}.
\begin{proposition}\label{prop:exp_moments_max_chi_n}
Let $\{\chi_n: n\in \N_0\}$ be a branching random walk with cluster distribution $\chi$ which satisfies~\eqref{eq:def_varphi}. Let $t\geq 0$.
\begin{enumerate}
\item If $\varphi'(t)>\beta_0$, then $c(t):=\lim_{n\to\infty} e^{-\varphi(t)n} \E [e^{t \max \chi_n}]$ exists in $(0,\infty)$.
\item If $\varphi'(t)<\beta_0$, then $c(t):=\lim_{n\to\infty} e^{-\varphi(t)n} \E [e^{t \max \chi_n}]=0$.
\end{enumerate}
\end{proposition}
\begin{remark}
We conjecture that for $\varphi'(t)=\beta_0$ the limit is $0$. The proof in the case of the binary branching Brownian motion will be given below; see Eqn.~\eqref{eq:exp_moment_boundary}.
\end{remark}
The next proposition shows that the maximal particle of the point process $\Pi_{e_{\lambda}}^{\chi}$ has a shifted Gumbel distribution.
\begin{proposition}\label{prop:max_equilibrium_gumbel}
Assume that $\varphi(\lambda)=0$, $\lambda>0$, $\varphi'(\lambda)>0$. Let $c(\lambda)$ be as in Proposition~\ref{prop:exp_moments_max_chi_n}. Then,
$$
\P[\max \Pi_{e_{\lambda}}^{\chi} \leq z]=\exp\left\{-\frac{c(\lambda)}{\lambda}e^{-\lambda z}\right\},\;\;\; z\in\R.
$$
\end{proposition}

\subsection{Equilibrium states and their family structure}\label{subsec:equilibrium_and_family}
A point process $\Xi$ on $\R$ is called equilibrium state for the cluster distribution $\chi$ if the branching dynamics starting at $\Xi$ is stationary in time. With other words, if $\sum_{i=1}^{\infty}\delta_{V_i}$ is a realization of $\Xi$ and $\chi^{(i)}$, $i\in\N$, are copies of the point process $\chi$ (all objects being independent), then the point process $\sum_{i=1}^{\infty} T_{V_i}\chi^{(i)}$ has the same law as $\Xi$. Clearly, the point process $\Pi_{e_{\lambda}}^{\chi}$ is an equilibrium state  for every $\lambda\in K_{\text{st}}$, where $K_{\text{st}}=\{\lambda\in\R: \varphi(\lambda)=0, \lambda\varphi'(\lambda)>0\}$.  Moreover, let $\Pi_{ce_{\lambda}}^{\chi}$ be the limiting point process for the branching dynamics started off the Poisson point process with intensity $ce_{\lambda}$, where $c>0$ and $\lambda\in K_{\text{st}}$. Then, any mixture of the form
\begin{equation}\label{eq:mixture}
\Xi(\cdot)=\int_{(0,\infty)\times K_{\text{st}}} \Pi_{c e_{\lambda}}^{\chi}(\cdot) \mu(dc,d\lambda),
\end{equation}
where  $\mu$ is a finite measure  on the space $(0,\infty)\times K_{\text{st}}$,  is an equilibrium state. The next theorem describes the set of equilibrium states in the non-boundary case. A related result for systems without branching can be found~\cite{aizenman_ruzmaikina}. Recall that a measure $J$ on $\R$ is called non-lattice if it is not concentrated on an arithmetic progression of the form $a\Z+b$, where $a,b\in\R$.
\begin{theorem}\label{theo:equilibrium_states}
Let $\chi$ be a cluster distribution satisfying~\eqref{eq:def_varphi} and having a non-lattice intensity measure $J$.   Assume that the equation $\varphi(\lambda)=0$ has two different solutions.
Then, any equilibrium state $\Xi$ whose intensity measure is finite on bounded sets can be represented as a mixture of the form~\eqref{eq:mixture}.
\end{theorem}


Consider a branching dynamics $\{\tilde \pi_n:n\in\N_0\}$ starting off a point process $\Pi_{e_{\lambda}}^{\chi}$. That is, $\tilde \pi_n=\sum_{i=1}^{\infty} T_{V_i}(\chi_n^{(i)})$, where $\sum_{i=1}^{\infty}\delta_{V_i}$ is a realization of $\Pi_{e_{\lambda}}^{\chi}$ and $\{\chi_n^{(i)}: n\in\N_0\}$, $i\in\N$, are copies of the branching random walk $\{\chi_n:n\in\N_0\}$, all objects being independent.  Being stationary, the $\MMM$-valued process $\{\tilde \pi_n:n\in\N_0\}$ admits a two-sided stationary extension denoted by $\{\tilde \pi_n: n\in\Z\}$.
Let $\mathcal F_{\leq n}$ be the $\sigma$-algebra generated by the random variables $\tilde \pi_k(B)$, where $B\subset \R$ is a bounded interval and $k=n, n-1,\ldots$. The $\sigma$-algebra of infinite past $\mathcal F_{-\infty}$ is defined by  $\mathcal F_{-\infty}=\cap_{n\in\Z}\mathcal F_{\leq n}$.
\begin{proposition}\label{prop:k_property}
Assume that $\varphi(\lambda)=0$, $\lambda\varphi'(\lambda)>0$.
Then, the $\sigma$-algebra $\mathcal F_{-\infty}$ is trivial, that is, every event in $\mathcal F_{-\infty}$ has probability $0$ or $1$. As a consequence, the process $\{\tilde \pi_n: n\in\Z\}$ is ergodic and mixing.
\end{proposition}


Clearly, the union of two independent point processes distributed as $\Pi_{c_1e_{\lambda}}^{\chi}$ and $\Pi_{c_2e_{\lambda}}^{\chi}$ has the same law as $\Pi_{(c_1+c_2)e_{\lambda}}^{\chi}$, $c_1,c_2>0$. It follows that the point process $\Pi_{e_{\lambda}}^{\chi}$ is infinitely divisible; see~\cite[\S~1.6]{matthes_etal_book}.
Moreover, the point process $\Pi_{e_{\lambda}}^{\chi}$ is superposable in the following sense; see~\cite{maillard}. Recall that $T_{u}$ denotes a translation by $u\in \R$ acting on the space $\MMM$ of discrete point configurations on $\R$. Let $\Pi'$ and $\Pi''$ denote two independent copies of $\Pi_{e_{\lambda}}^{\chi}$. Then, for every $u_1,u_2,u\in\R$ such that $e^{\lambda u_1}+e^{\lambda u_2}=e^{\lambda u}$ we have the following equality in distribution:
\begin{equation}\label{eq:maillard}
T_{u_1}\Pi'+T_{u_2}\Pi''\eqdistr T_u\Pi_{e_{\lambda}}^{\chi}.
\end{equation}
To see that~\eqref{eq:maillard} holds, note that the Poisson point process with intensity $e_{\lambda}$ satisfies~\eqref{eq:maillard} and that superposability is preserved under clustering and taking weak limits.  Superposable point processes have been characterized in~\cite{maillard}. As a consequence of~\cite{maillard}, we obtain
\begin{proposition}\label{prop:maillard}
There is a unique in law point process $\Gamma$ (depending on $\lambda$) satisfying $\max \Gamma=0$ a.s.\ and a constant $c>0$ such that if $\Gamma_i$, $i\in\N$, are independent copies of $\Gamma$ and independently, $\sum_{i=1}^{\infty}\delta_{W_i}$ is a Poisson point process on $\R$ with intensity $ce_{\lambda}$, then
\begin{equation}\label{eq:maillard_prop}
\Pi_{e_{\lambda}}^{\chi} \eqdistr \sum_{i=1}^{\infty} T_{W_i}\Gamma_i.
\end{equation}
\end{proposition}
It is convenient to think of~\eqref{eq:maillard_prop} as of the decomposition of the point process $\Pi_{e_\lambda}^{\chi}$ into the families $T_{W_i}\Gamma_i$, $i\in\N$; see~\cite[\S~1.6]{main_book}. If we would continue the branching dynamics starting off $\Pi_{e_\lambda}^{\chi}$ back into the past, then any two particles alive at time $0$ belong to the same family $T_{W_i}\Gamma_i$ if and only if they have a common ancestor somewhere in the past.
\begin{proposition}\label{prop:family_finite}
Assume that $\varphi(\lambda)=0$, $\lambda>0$, $\varphi'(\lambda)>0$. If the cluster distribution $\chi$ is critical, assume additionally that $\E[\chi(\R)^2]$ is finite and $\P[\chi(\R)=1]<1$.
Then, the point process $\Gamma$ consists  of finitely many particles a.s.\ if and only if the cluster distribution $\chi$ is subcritical.  
\end{proposition}
We have already noticed that even in the persistent branching dynamics the offsprings of any particle converge to $-\infty$ in a uniform way  or die out. The next proposition says that a similar conclusion applies to the progeny of any family. Let $\varrho_n^{(i)}$ be the progeny of the family $T_{W_i}\Gamma_i$ at time $n$ in the stationary branching dynamics starting off $\Pi_{e_{\lambda}}^{\chi}$.
\begin{proposition}\label{prop:family_drift}
Assume that $\varphi(\lambda)=0$, $\lambda>0$, $\varphi'(\lambda)>0$. If the cluster distribution $\chi$ is critical, assume additionally that $\E[\chi(\R)^2]$ is finite.
If $\chi$ is critical or supercritical, then
$$
\P[\forall i\in\N:  \lim_{n\to\infty} \max \varrho_n^{(i)}=-\infty]=1.
$$
\end{proposition}
In particular, every family produces a.s.\ only finitely many leaders and finitely many particles above any fixed level. Also, it follows that there are a.s.\  infinitely many families (at time $0$) which manage to produce a leading particle at some time in the sequel. Let us stress that the convergence in Proposition~\ref{prop:family_drift} is not uniform in $i$, since otherwise the whole branching system would drift to $-\infty$ like in the case of local extinction.

\begin{remark}
If the cluster distribution $\chi$ allows only for integer-valued displacements of the particles, then it is natural to consider a branching dynamics on $\Z$ starting off a Poisson point process on $\Z$ with intensity $\sum_{u\in\Z}e^{-\lambda u} \delta_{u}$. Our results have straightforward analogues in this settings. The non-lattice condition in Theorem~\ref{theo:equilibrium_states} should be replaced by the assumption that $J$ is aperiodic.
\end{remark}

\subsection{The boundary case}\label{sec:boundary}
We consider a system of particles starting off a Poisson point process on $\R$ with intensity $e^{-\lambda u}du$, where $\lambda>0$, and performing independent branching Brownian motions with drift $-c$, where $c>0$. This means that each particle performs a standard Brownian motion with drift $-c$; after an exponential time it splits into two particles which behave in the same way as the original particle independently of each other. Branching Brownian motion is a continuous-time process, but its restriction to integer times belongs to the class of processes considered here.
It is readily checked that $\varphi(\lambda)=(\lambda^2/2)-c\lambda+1$. For $c>\sqrt 2$ the equation $\varphi(\lambda)=0$ has two different solutions $\lambda_{\pm}=c\pm \sqrt{c^2-2}$. By Theorem~\ref{theo:main} the measure $e_{\lambda_+}$ corresponds to a persistent branching dynamics, while the measure $e_{\lambda_-}$ corresponds to a branching dynamics which becomes locally extinct. In the boundary case $c=\sqrt 2$ there is only one solution $\lambda=\sqrt 2$. \citet{lalley_sellke} conjectured that in this case the branching dynamics should be persistent. We will prove that this is not the case.
\begin{theorem}\label{theo:boundary_unstable}
Consider a system of independent branching Brownian motions with drift $-\sqrt 2$ starting off the Poisson point process on $\R$ with intensity $e_{\sqrt 2}$. Denote by $\pi_n$ the point processes formed by the particles at time $n$.  Then, $\pi_n$ converges as $n\to\infty$ to the empty process.
\end{theorem}
However, the second solution of $\varphi(\lambda)=0$ in the case $c=\sqrt 2$ is not gone lost completely. An easy computation shows that the \textit{signed} measure $\nu(du)=ue^{-\sqrt 2 u}du$ is a solution of the convolution equation~\eqref{eq:deny} in the case $c=\sqrt 2$.  It has been recently shown in~\cite{arguin_etal3} that the branching dynamics starting off a Poisson point process with intensity $-ue^{-\sqrt 2 u}1_{u<0}du$ converges to a non-trivial limiting point process.

\section{Proof of Theorem~\ref{theo:main}}

\subsection{Two lemmas on large deviations}
Recall that
$I(z)=\sup_{t\in\R} (zt -\varphi(t))$ is the Legendre--Fenchel transform of $\varphi$.
We need the following simple Chernoff-type estimate.
\begin{lemma}\label{lem:large_dev}
Let $\{\chi_n: n\in\N_0\}$ be a branching random walk with cluster distribution $\chi$ satisfying~\eqref{eq:def_varphi}. Then, for every $n\in\N$ and $a\geq n\varphi'(0)$,
\begin{equation}\label{eq:def_I}
\P[\max \chi_n \geq a]\leq e^{-n I(a/n)}.
\end{equation}
\end{lemma}
\begin{proof}
Let $J^{[n]}=J*\ldots*J$ be the $n$-th convolution power of $J$. Note that $J^{[n]}$ is the intensity of $\chi_n$. We have, for every $t\geq 0$,
\begin{align*}
\P[\max \chi_n\geq a]
\leq J^{[n]}([a,\infty))
\leq \int_{\R}e^{t (u-a)}J^{[n]}(du) =e^{-n(\frac{a}{n}t- \varphi(t))}.
\end{align*}
The proof is completed by noting that $\sup_{t\geq 0}(\frac an t-\varphi(t))=I(\frac an )$, where the supremum can be taken over $t\geq 0$ since $\frac an\geq \varphi'(0)$.
\end{proof}

\begin{lemma}\label{lem:exp_moment_max}
Let $\{\chi_n: n\in\N_0\}$ be a branching random walk with a supercritical cluster distribution $\chi$ satisfying~\eqref{eq:def_varphi}. If $t\geq 0$ is such that $\varphi'(t)<  \beta_0$, then
$$
\lim_{n\to\infty} \frac 1n \log \E [e^{t \max \chi_n}]=
\beta_0t.
$$
\end{lemma}
\begin{remark}
If $\varphi'(t)>\beta_0$, then the limit is equal to $\varphi(t)$; see Proposition~\ref{prop:exp_moments_max_chi_n}.
\end{remark}
\begin{proof}
Write $M_n=\max \chi_n$. By~\eqref{eq:biggins},  we have $\lim_{n\to\infty}\frac 1n M_n=\beta_0$ a.s.\ on survival. It follows in a straightforward way that
$$
\liminf_{n\to\infty}\frac 1n \log \E [e^{t M_n}]\geq \beta_0 t.
$$
We prove the converse inequality. Trivially, $\E \left[e^{tM_n}1_{M_n\leq \beta_0 n}\right]\leq e^{\beta_0 t n}$.
Fix $\eps>0$. Then, by Lemma~\ref{lem:large_dev},
$$
\E \left[e^{tM_n}1_{M_n>\beta_0 n}\right]
\leq \sum_{z\in\beta_0+\eps \N_0} e^{t(z+\eps)n}\P[M_n\geq zn]
\leq e^{\eps t n} \sum_{z\in\beta_0+\eps \N_0} e^{(tz-I(z))n}.
$$
By convexity and since $I(\beta_0)=0$ we have $I(z)\geq I'(\beta_0)(z-\beta_0)$. Note also that $t=I'(\varphi'(t))<I'(\beta_0)$ and $\beta_0<0$. It follows that
$$
\E \left[e^{tM_n}1_{M_n>\beta_0n}\right]
\leq
e^{\eps t n}e^{\beta_0I'(\beta_0) n} \sum_{z\in\beta_0+\eps \N_0} e^{(t-I'(\beta_0))zn}
\leq C e^{(\beta_0+\eps)tn}.
$$
Bringing everything together, we obtain the statement of the lemma.
\end{proof}

\subsection{Kallenberg's stability criterion}\label{subsec:kallenberg}
Our proof of the persistence part of Theorem~\ref{theo:main} is based on the backward tree construction introduced by \citet{kallenberg} in the context of critical branching dynamics on $\R^d$ starting off a homogeneous Poisson point process. This approach has been generalized in~\cite[\S~2.4]{main_book} (see also~\cite{wakolbinger_book}) to arbitrary branching dynamics on Polish spaces. 

Recall that we consider a branching dynamics on $\R^d$ with spatially homogeneous cluster distribution $\chi$ having intensity $J$. We assume that $\lambda\in\R^d$ is such that $\varphi(\lambda)=0$. Let $o$ be an individuum which is alive at time $0$ and is located at the origin in $\R^d$. The idea of Kallenberg's method is to trace back the history of the individual $o$. First, we define the set of the direct ancestors of $o$.  Let $D$ be a probability measure on $\R^d$ defined by
\begin{equation}\label{eq:def_D}
D(B)=\int_{B} e^{\langle \lambda, u\rangle} J(du), 
\end{equation}
where $B\subset\R^d$ is Borel. Since $\varphi(\lambda)=0$, $D$ is indeed a probability measure.
Let $\{\xi_{n}:n\in \N\}$, be independent $\R^d$-valued random variables with distribution $D$. Let $S_{n}=\xi_{1}+\ldots+\xi_{n}$, $S_0=0$, be the corresponding random walk. It is convenient to think of $-S_{n}$ as of the position of the direct ancestor of $o$ which is alive at time $-n$.  So,  $-S_{1}$ is the position of the father of $o$ alive at time $-1$, $-S_{2}$ is the position of the grandfather of $o$ alive at time $-2$, etc.  Now we define the set of relatives of $o$ alive at time $0$, i.e., the set of its brothers, cousins, etc. Let $\{\chi_{(z)}:z\in\R^d\}$ be the system of Palm measures of the point process $\chi$; see~\cite[\S~1.8]{main_book} and~\cite[Ch.~8,9]{matthes_etal_book}. Denote by $\{\chi_{(z)}^!:z\in\R^d\}$ the reduced Palm measures. Recall that $\chi_{(z)}^!$ is obtained from $\chi_{(z)}$ by removing the point $z$.
Conditioned on  $\xi_1,\xi_2,\ldots$ let $\kappa_n=\sum_{k=1}^{k_n} \delta_{z_{k:n}}$, $n\in\N$, be independent point processes such that $\kappa_n$ has the same law as $\chi_{(\xi_n)}^{!}$. Think of $-S_n+z_{k:n}$, $k=1,\ldots,k_n$, as of the locations of the brothers (alive at time $-(n-1)$) of the ancestor of $o$ located at $-S_{n-1}$. Let also $\{\chi_l^{(k,n)}: l\in\N_0\}$, where $n\in\N$ and $k=1,\ldots,k_n$, be independent copies of the original branching random walk $\{\chi_l: l\in\N_0\}$ generated by the cluster distribution $\chi$. Think of $\chi_l^{(k,n)}$ as of the point process describing the positions of the offsprings in the $l$-th generation of the individuum located at time $-(n-1)$ at $-S_n+z_{k:n}$. Then, we define a point process $\rho_n$ by
\begin{equation}\label{eq:def_rho_n}
\rho_n=\sum_{k=1}^{k_n} T_{-S_{n}+z_{k:n}} \left(\chi_{n-1}^{(k,n)}\right).
\end{equation}
Think of $\rho_n$ as of the point process describing the positions of the relatives of $o$ which are alive at time $0$ and have the same ancestor at time  $-n$ as $o$. Let $\rho_0=\delta_0$. The next theorem is a specification of~\cite[Proposition~2.4.3]{main_book} to our setting.
\begin{theorem}[Kallenberg's stability criterium]\label{theo:kallenberg}
The branching dynamics starting off a Poisson point process with intensity $e_{\lambda}$ is persistent if and only if the measure $\rho:=\sum_{n=0}^{\infty}\rho_n$ is locally finite with probability $1$.
\end{theorem}

\subsection{Proof of Part~\ref{p:main_1} of Theorem~\ref{theo:main}}\label{sec:proof_main_p1}
First we introduce some notation which is needed to handle the case $d\geq 2$ and may be ignored in the case $d=1$. Without loss of generality we may assume that $\lambda=(l,0,\ldots,0)$ for some $l>0$, otherwise, we can rotate the coordinates.  We denote by $P:\R^d\to\R$ the projection operator $P(l_1,\ldots,l_d)=l_1$.
We may project the cluster distribution $\chi$ to $\R$ via $P$. The resulting cluster distribution $\chi_P$ has intensity  $J_P(B)=J(P^{-1}B)$, where $B\subset \R$ is Borel, and the log-Laplace transform of $J_P$ is given by $\varphi_P(t)=\varphi(t,0,\ldots,0)$, $t\in\R$. Let $I_P(z)=\sup_{t\in\R}(zt-\varphi_P(t))$ be the Legendre--Fenchel transform of $\varphi_P$. We have $\varphi_P(l)=0$. Also, the assumption $\langle \lambda, \nabla\varphi(\lambda)\rangle>0$ implies that $\varphi_P'(l)>0$.

By Theorem~\ref{theo:kallenberg} it suffices to verify that $\rho(A)$ is finite a.s.\ for $A$ being a half-space  set of the form $[a,\infty)\times \R^{d-1}$, $a\in\R$. Denote by $\AAA_{\xi,\kappa}$ the $\sigma$-algebra generated by the random variables $\{\xi_n: n\in\N\}$ and the point processes $\{\kappa_n:n\in\N\}$.  We have, by~\eqref{eq:def_rho_n}, Lemma~\ref{lem:large_dev} and the convexity of the function $I_P$,
\begin{align*}
\lefteqn{\P[\rho_n(A)>0|\AAA_{\xi,\kappa}]}\\
&\leq \sum_{k=1}^{k_n} \P\left[\max P(\chi_{n-1}^{(k,n)}) \geq  P(S_{n})-P(z_{k:n})+a\right]\\
&\leq \sum_{k=1}^{k_n} \exp\left\{{-(n-1) I_P\left(\frac{P(S_n)-P(z_{k:n})+a}{n-1}\right)}\right\}\\
&\leq \sum_{k=1}^{k_n} \exp\left\{{-(n-1) I_P\left(\frac{P(S_n)}{n-1}\right)+ I'_P\left(\frac{P(S_n)}{n-1}\right)(P(z_{k:n})-a)}\right\}.
\end{align*}
The expectation of $\xi_1$ is given by $\E[\xi_1]=\int_{\R^d} uD(du)=\nabla\varphi(\lambda)$. By the law of large numbers,
$$
\frac{P(S_{n})}{n}\to P(\nabla\varphi(\lambda))=\varphi'_P(l)>0 \;\;\;\text{ a.s.}
$$
By the properties of the Legendre--Fenchel transform, $I_P(\varphi'_P(l))=l \varphi'_P(l)$ and $I'_P(\varphi'_P(l))=l$.
Take $\eps>0$. It follows that with probability $1$ we have for sufficiently large $n$,
$$
I_P\left(\frac{P(S_n)}{n-1}\right)>l \varphi'_P(l)-\eps,\;\;\;
l-\eps<I_P'\left(\frac{P(S_n)}{n-1}\right)<l+\eps.
$$
Thus, we have
\begin{equation}\label{eq:proof_kall_rho_est}
\P[\rho_n(A)>0|\AAA_{\xi,\kappa}]
\leq
C e^{-(l \varphi'_P(l)-\eps)n} Y_n,
\end{equation}
where $\{Y_n: n\in\N\}$ are random variables defined by
$$
Y_n= \sum_{k=1}^{k_n}(e^{(l+\eps) P(z_{k:n})}+e^{(l-\eps) P(z_{k:n})})
=\int_{\R^d}(e^{\langle \lambda_{+\eps},u\rangle}+e^{\langle \lambda_{-\eps},u\rangle})\kappa_n(du)
$$
and $\lambda_{\pm\eps}=(l\pm\eps,0,\ldots,0)$. We will show that $\E [Y_n]$ is finite.
Define a function $f:\MMM\to [0,+\infty]$ (where $\MMM$ is  the space of all locally finite counting measures on $\R^d$) by $f(\beta)=\int_{\R^d} (e^{\langle\lambda_{+\eps}, u\rangle}+e^{\langle\lambda_{-\eps}, u\rangle})\beta(du)$. Then,
\begin{align*}
\E [Y_n]
=\int_{\R^d} \E [f(\chi_{(u)}^{!})] J(du)
\leq \int_{\R^d} \E[ f(\chi_{(u)})] J(du)
=\E [f(\chi)]\cdot \E[\chi (\R^d)],
\end{align*}
where the last step follows from the definition of the Palm measure. It follows that
\begin{equation}\label{eq:exp_Y_n}
\E [Y_n]= (e^{\varphi(\lambda_{+\eps})}+e^{\varphi(\lambda_{-\eps})})e^{\varphi(0)}.
\end{equation}
Recall that in Part~\ref{p:main_1} of Theorem~\ref{theo:main} we assume that $l \varphi'_P(l)>0$. It follows from~\eqref{eq:proof_kall_rho_est} and~\eqref{eq:exp_Y_n} that
$$
\sum_{n=1}^{\infty} \P[\rho_n(A)>0]
=
\sum_{n=1}^{\infty} \E [\P[\rho_n(A)>0|\AAA_{\xi,\kappa}]]<\infty.
$$
By the Borel--Cantelli lemma, $\rho(A)=\sum_{n=0}^{\infty}\rho_n(A)$ is finite a.s.  By Theorem~\ref{theo:kallenberg} the proof is completed. 

\subsection{Proof of Part~\ref{p:main_2} of Theorem~\ref{theo:main}}
\begin{proof}[Proof of Proposition~\ref{prop:unstab_speed}]
Recall the definition of $\pi_n$ in~\eqref{eq:def_pi_n}. The point process  $\pi_n^*:=\sum_{i=1}^{\infty}\delta_{U_{i}+\max \chi_n^{(i)}}$ is a Poisson point process on $\R$ with intensity $\E [e^{\lambda \max\chi_n}]\cdot  e^{-\lambda u}du$.
For every $\beta>\beta_0$ we have
\begin{align}\label{eq:proof_prop_speed_unstab}
\sum_{n=1}^{\infty} \P[\max \pi_n>\beta n]
\leq \sum_{n=1}^{\infty} \E [\pi_n^*((\beta n,\infty))]
= \frac 1 {\lambda} \sum_{n=1}^{\infty} (e^{-\lambda \beta n} \E[ e^{\lambda \max \chi_n}])
<\infty,
\end{align}
where the last step follows from Lemma~\ref{lem:exp_moment_max} (note that $\varphi'(\lambda)<\beta_0$). By the Borel--Cantelli lemma, we have
$
\limsup_{n\to\infty} \frac 1n\max\pi_n\leq \beta_0.
$
The converse inequality follows immediately from~\eqref{eq:biggins}.
\end{proof}

\begin{proof}[Proof of Part~\ref{p:main_2} of Theorem~\ref{theo:main}]
 Let first $d=1$. For concreteness, assume that $\lambda>0$. Then, $\varphi'(0)<0$,  $\beta_0<0$ and Proposition~\ref{prop:unstab_speed} implies that $\lim_{n\to\infty} \max \pi_n=-\infty$ a.s. It follows that $\pi_n$ converges weakly to the empty process.

We consider the case $d\geq 2$. Without restriction of generality let $\lambda=(l,0,\ldots,0)$ with $l>0$.
Let $A\subset \R^{d}$ be a bounded Borel set. We will show that with probability $1$,  $\pi_n(A)=0$ for sufficiently large $n$.
Let $E_n=\R\times [-n^2, n^2]^{d-1}$ and $\bar E_n=\R^d\backslash E_n$. Let $N_n$ (respectively, $\bar N_n$) be a random variable counting the number of particles of the point process $\pi_n$ which are in the set $A$ and whose ancestor at time $0$ was in the set $E_n$ (respectively, in $\bar E_n$). Clearly, $\pi_n(A)=N_n+\bar N_n$. After elementary transformations we obtain
\begin{align*}
\E [\bar N_n]
=
\int_{\bar E_n} e^{-\langle\lambda, u\rangle} J^{[n]}(A-u)du
=\int_{A} e^{-\langle\lambda, a\rangle}D^{[n]}(a-\bar E_n)da,
\end{align*}
where $D^{[n]}=D*\ldots*D$ is the $n$-th convolution power of the probability measure $D$  defined in~\eqref{eq:def_D}.
By the Cram\'er--Chernoff bound there is $c>0$ such that $D^{[n]} (a-\bar E_n)=O(e^{-cn})$ uniformly in $a\in A$. Hence, $\P[\bar N_n>0]\leq \E [\bar N_n]= O(e^{-cn})$. By the Borel--Cantelli lemma, $\bar N_n=0$ for all but finitely many $n$ a.s.

Let us show that with probability $1$, $N_n=0$ for all but finitely many $n$.
Consider the progeny of all particles which start at time $0$ in the cylinder set $\R\times [0,1]^{d-1}$. Let us project these particles onto the first coordinate $\R$ via the projection operator $P$; see Section~\ref{sec:proof_main_p1}.  We obtain a branching dynamics $\{\pi_n^{(P)}: n\in\N_0\}$ on $\R$ starting off a Poisson point process with intensity $e^{-lx}$. By Proposition~\ref{prop:unstab_speed} (see also~\eqref{eq:proof_prop_speed_unstab}) there are $\beta<0$, $c>0$ such that $\P[\max \pi_n^{(P)} >\beta n]=O(e^{-cn})$. Divide the set $E_n$ into $(2n^2)^{d-1}$ sets which are translates of the set $\R\times [0,1]^{d-1}$ and apply to each of them the above considerations. Let $R_n$ be the largest among the first coordinates of the offsprings in the $n$-th generation of the particles which started in $E_n$. It follows that $\P[R_n>\beta n]=O(n^{2d-2}e^{-cn})=O(e^{-cn/2})$. By the Borel--Cantelli lemma, $R_n\to -\infty$ a.s. It follows that with probability $1$, $N_n=0$ for all but finitely many $n$.  This completes the proof.
\end{proof}

\subsection{Proof of Theorem~\ref{theo:boundary_unstable}}
Let $\chi_n$ be the point process formed by the particles at time $n$ in a single branching Brownian motion with drift $-\sqrt 2$. Write $M_n=\max\chi_n$. It suffices to show that
\begin{equation}\label{eq:exp_moment_boundary}
\lim_{n\to\infty}\E [e^{\sqrt 2 M_n}]=0.
\end{equation}
Fix $\eps>0$. Note that the expected number of particles at time $n$  in the branching Brownian motion is $e^n$. Denoting by $\bar\Phi$ the tail of the standard normal distribution function we have a trivial estimate
$$
\P\left[M_n\geq y\right] \leq \E [\chi_n([y,\infty))]=e^n \bar \Phi\left(\frac{y+\sqrt 2 n}{\sqrt n}\right)\leq \frac C {\sqrt n  e^{\sqrt 2 y}} e^{-\frac{y^2}{2n}}.
$$
It follows that for a sufficiently large $c>0$,
\begin{align*}
\E \left[e^{\sqrt 2 M_n}1_{M_n\geq [c\sqrt n]}\right]
&\leq C\sum_{k=[c\sqrt n]}^{\infty} e^{\sqrt 2 k} \P\left[M_n\geq k\right]\\
&\leq \frac C{\sqrt n}\sum_{k=[c\sqrt n]}^{\infty}  e^{-\frac{k^2}{2n}}\\
&\leq C\int_{c-\frac 2 {\sqrt n}}^{\infty}e^{-\frac{s^2}{2}}ds\\
&\leq \eps.
\end{align*}
Write $l_n=\frac 3{2\sqrt 2}\log n$. By Proposition~3 of~\citet{bramson} there is $B=B(c)$ such that for all $-l_n\leq y \leq c\sqrt n$, we have the following estimate
$$
\P\left[M_n\geq y\right]\leq  \frac {B}{n^{3/2}e^{\sqrt 2 y}}.
$$
This is stated in~\cite{bramson} for $c=1$ only, but the proof is valid for every $c>0$.
It follows that
$$
\E \left[e^{\sqrt 2 M_n}1_{-[l_n]\leq M_n\leq [c\sqrt n] }\right]
\leq \sum_{k=-[l_n]}^{[c\sqrt n]} e^{\sqrt 2(k+1)}\P\left[M_n\geq k\right]
=O\left(\frac 1n\right).
$$
Trivially, we have
$
\lim_{n\to\infty}\E \left[e^{\sqrt 2M_n}1_{M_n\leq -l_n}\right]=0.
$
Bringing everything together, we obtain~\eqref{eq:exp_moment_boundary}.

\section{Proof of Propositions~\ref{prop:leader_stable}, \ref{prop:exp_moments_max_chi_n}, \ref{prop:max_equilibrium_gumbel}, \ref{prop:family_drift}}

\subsection{Proof of Propositions~\ref{prop:exp_moments_max_chi_n} and~\ref{prop:max_equilibrium_gumbel}}
Assume first that $\varphi(t)=0$ and $\varphi'(t)>\beta_0$, which, in fact, implies that $t\varphi'(t)>0$. Then, the branching dynamics $\{\pi_n: n\in\N_0\}$ starting off a Poisson point process with intensity $e_{t}$ is persistent by Theorem~\ref{theo:main}. In particular, it follows that
$\lim_{n\to\infty} \P[\pi_n((z,+\infty))=0]$ exists in $(0,1)$ for at least one $z\in \R$. On the other hand,
\begin{equation}\label{eq:proof_gumbel}
\P [\max \pi_n\leq z]=\P[\pi_n((z,+\infty))=0]=\exp\left\{-\frac 1t \E [e^{t \max \chi_n}]\cdot e^{-tz}\right\}.
\end{equation}
Hence, $c(t)=\lim_{n\to\infty}\E [e^{t \max \chi_n}]$ exists in $(0,\infty)$.

If $\varphi(t)\neq 0$ and $\varphi'(t)>\beta_0$, we consider another branching random walk $\{\tilde \chi_n: n\in\N_0\}$ obtained by adding to $\{\chi_n: n\in\N_0\}$ a drift $-\varphi(t)n/t$, that is $\tilde \chi_n= T_{-\varphi(t)n/t}(\chi_n)$. The log-Laplace transform $\tilde \varphi$ of $\tilde \chi_1$ satisfies $\tilde \varphi(t)=0$ and $\tilde \varphi'(t)>\beta_0-\frac {\varphi(t)}{t}=\tilde \beta_0$, where $\tilde \beta_0$ is the analogue of $\beta_0$ for the function $\tilde I(z)=\sup_{s\in\R}(sz-\tilde \varphi(s))$. By the above,
$$
\lim_{n\to\infty} e^{-\varphi(t)n} \E [e^{t \max \chi_n}]
=\lim_{n\to\infty} \E [e^{t  \max\tilde \chi_n}]\in (0,\infty).
$$
This completes the proof of the first part of Proposition~\ref{prop:exp_moments_max_chi_n}. The second part follows from Lemma~\ref{lem:exp_moment_max}. Proposition~\ref{prop:max_equilibrium_gumbel} follows by taking the limit $n\to\infty$ in~\eqref{eq:proof_gumbel}.

\subsection{Proof of Proposition~\ref{prop:leader_stable}}
The idea of the proof is as follows. At time $0$, the density of particles at $-un$ is $e^{\lambda un}$, and the probability that a particle at $-un$ will produce at time $n$ an offspring, say, on the positive half-axis is approximately $e^{-nI(u)}$; cf.\ Lemma~\ref{lem:large_dev}. Thus, the ancestor at time $0$ of the particle which is the leader at time $n$ is most likely to be found at position approximately $-un$, where $u$ is such that $\lambda u- I(u)$ is maximal. The maximizer is given by $u=\varphi'(\lambda)$, whence the statement.

Let us provide a rigorous argument justifying this. Fix some $\eps>0$ and write $\alpha_{\pm}(\eps)=\varphi'(\lambda)\pm\eps$.  We estimate the probability $P_n^+(\eps)$ that some offspring of some particle which started at time $0$ below the level  $-\alpha_+(\eps)n$ will be above the level $-\sqrt n$ at time $n$:
\begin{align*}
P_n^+(\eps)&:=
\P\left[\max_{\genfrac{}{}{0pt}{1} {i\in\N} {U_i<-\alpha_+(\eps)n}  }(U_i+ \max  \chi_n^{(i)})>-\sqrt n\right]\\
&\leq
\int_{\alpha_+(\eps)n}^{\infty} e^{\lambda u}\P[\max \chi_n\geq u-\sqrt n]du\\
&=
n e^{\lambda \sqrt n} \int_{\alpha_+(\eps)-\frac 1 {\sqrt n}}^{\infty} e^{\lambda n v}\P[\max\chi_n\geq nv]dv\\
&\leq
n e^{\lambda \sqrt n} \int_{\alpha_+(\eps/2)}^{\infty} e^{(\lambda  v-I(v))n} dv,
\end{align*}
where we have used a substitution $u=nv+\sqrt n$ and Lemma~\ref{lem:large_dev}. The last inequality holds if $n$ is sufficiently large.
The function $g(v):=\lambda v-I(v)$ is concave and we have $g(\varphi'(\lambda))=0$, $g'(\varphi'(\lambda))=0$.  The global maximum of $g$ is attained at $v=\varphi'(\lambda)$.
By concavity there are $c_1>0$, $c_2>0$ such that
$g(v)\leq -c_1-c_2(v-\alpha_+(\eps/2))$ for all $v\geq \alpha_+(\eps/2)$. Hence,
$$
P_n^+(\eps)
\leq
n e^{\lambda \sqrt n} e^{-c_1n} \int_{\alpha_+(\eps/2)}^{\infty} e^{-c_2  (v-\alpha_+(\eps/2))n} dv=O(e^{-\frac 12 c_1 n}).
$$
Analogous considerations apply to $P_n^-(\eps)$, the probability that some offspring of some particle which started at time $0$ above the level $-\alpha_-(\eps)n$ will be above the level $-\sqrt n$ at time $n$. Hence,
$
\sum_{n=1}^{\infty} P_n^+(\eps)<\infty
$
and
$
\sum_{n=1}^{\infty} P_n^-(\eps)<\infty
$.
By the Borel--Cantelli lemma, only finitely often a particle starting not in the interval $(-\alpha_+(\eps)n,-\alpha_-(\eps)n)$ will be able to produce a particle above the level $-\sqrt n$ at time $n$.
Also,  we have
$$
\sum_{n=1}^{\infty} \P[\max \pi_n<-\sqrt n]=\sum_{n=1}^{\infty} \exp\left\{- \frac 1{\lambda}\E [e^{\lambda \max\chi_n}] e^{\lambda \sqrt n}\right\}<\infty,
$$
where the last inequality holds since $\E[e^{\lambda \max\chi_n}]$ remains bounded  by Proposition~\ref{prop:exp_moments_max_chi_n}.
Hence, only finitely often the leading particle will be below the level $-\sqrt n$. The statement of the proposition follows.


\subsection{Proof of Propositions~\ref{prop:family_finite} and~\ref{prop:family_drift}}
We will use the notation introduced in Section~\ref{subsec:kallenberg}. If the cluster distribution is critical, we always assume that $\E[\chi(\R)^2]$ is finite.
\begin{lemma}\label{lem:rho_finite}
If the cluster distribution $\chi$ is  subcritical or $\chi(\R)=1$ a.s., then $\rho(\R)<\infty$ a.s. Otherwise, $\rho(\R)=\infty$ a.s.
\end{lemma}
\begin{proof}
In the subcritical case we have $\E[\chi(\R)]<1$ and hence,
$$
\E[\rho(\R)]=\sum_{n=0}^{\infty} \E [\rho_n(\R)]=1+\sum_{n=1}^{\infty} \E [k_n] \E[\chi_{n-1}(\R)]= 1+C \sum_{n=1}^{\infty} (\E[\chi(\R)])^{n-1}<\infty.
$$
It follows that $\rho(\R)<\infty$ a.s.
In the non-subcritical case we write
$$
\P[\rho_n(\R)>0]\geq \P[k_n\geq 1, \chi_{n-1}^{(1,n)}(\R)>0]=\P[k_n\geq 1]\cdot \P[\chi_{n-1}(\R)>0].
$$
If we assume that $\P[\chi(\R)=1]<1$, then $\P[k_n\geq 1]>0$ is a constant independent of $n$. Further, in the critical case we can find $c>0$ such that for all $n\in\N$, $\P[\chi_{n-1}(\R)>0]>c/n$; see~\cite[p.~19]{athreya_book}. (Here the finiteness of $\E[\chi(\R)^2]$ is needed). In the supercritical case, we even have $\P[\chi_{n-1}(\R)>0]>c$.
By the Borel--Cantelli lemma,  $\rho(\R)=\infty$ a.s.
\end{proof}

\begin{lemma}\label{lem:equiv_rho_gamma}
Let $A\subset \MMM$ be a Borel set such that if $\beta\in A$, then $T_u\beta\in A$ for every $u\in\R$.  Let $\varphi(\lambda)=0$, $\lambda>0$, $\varphi'(\lambda)>0$. Then, $\P[\Gamma\in A]=0$ if and only if $\P[\rho\in A]=0$
\end{lemma}
\begin{proof}
As already observed in Section~\ref{subsec:equilibrium_and_family}, the point process $\Pi_{e_{\lambda}}^{\chi}$ is infinitely divisible. Denote by $\Psi$ its canonical (L\'evy) measure; see~\cite[Ch.~2]{matthes_etal_book}. Note that $\Psi$ is a generally infinite measure on $\MMM$.
On the one hand, the cluster representation of $\Pi_{e_{\lambda}}^{\chi}$ given in Proposition~\ref{prop:maillard} implies that we have
$$
\Psi(A)=c\int_{\R} \P[T_u\Gamma\in A] e^{-\lambda u}du.
$$
Thus, $\P[\Gamma\in A]=0$ if and only if $\Psi(A)=0$.
On the other hand, by Theorem~2.4.4 of~\cite{main_book} we have
$$
\Psi_{(u)}(A)=\P[T_u\rho\in A], \;\;\, u\in\R,
$$
where $\{\Psi_{(u)}: u\in\R\}$ is the family of Palm measures of $\Psi$. Thus, $\P[\rho\in A]=0$ if and only if $\Psi_{(u)}(A)=0$ for all $u\in\R$. Finally, by the definition of the Palm measure we have for every Borel set $B\subset \R$,
$$
\int_{A}\beta(B)\Psi(d\beta)=\int_B \Psi_{(u)}(A)e_{\lambda}(du).
$$
Thus, $\Psi_{(u)}(A)=0$ for all $u\in\R$ if and only if $\Psi(A)=0$.
\end{proof}
\begin{proof}[Proof of Proposition~\ref{prop:family_finite}]
The proof follows from Lemma~\ref{lem:rho_finite} and Lemma~\ref{lem:equiv_rho_gamma} by taking $A$ to be the set of all point configurations $\beta\in\MMM$ such that $\beta(\R)=\infty$ (in the subcritical case)  or  $\beta(\R)<\infty$ (otherwise).
\end{proof}

\begin{lemma}\label{lem:palm_exp}
Let $\varphi(\lambda)=0$, $\lambda>0$, $\varphi'(\lambda)>0$. If  $\chi$ is critical or supercritical, then $\int_{\R} e^{ru}\Gamma(du)<\infty$ for every $r>r_0$, where $r_0\geq 0$ is the smallest solution of the equation $\varphi(r)=\varphi'(\lambda)r$.
\end{lemma}
\begin{proof}
By Lemma~\ref{lem:equiv_rho_gamma} we need to show that $\int_{\R} e^{ru} \rho(du)<\infty$ a.s.\ for all $r>r_0$.  Recall that $\AAA_{\xi}$ is the $\sigma$-algebra generated by $\{\xi_n:n\in\N\}$. Denote by $J_{(u)}$ the intensity of the point process $\chi_{(u)}$, $u\in\R$. Fix $\eps>0$.
 We have $S_n/n\to \varphi'(\lambda)$ by the law of large numbers.   It follows that with probability $1$, $S_n/n>\varphi'(\lambda)-\eps$ for sufficiently large $n$. Hence,
\begin{align}
\E \left[\int_{\R}e^{ru}\rho(du)|\AAA_{\xi}\right]
&=1+\sum_{n=1}^{\infty} \int_{\R} e^{ru}T_{-S_n}(J^{[n-1]}*J_{(\xi_n)})(du)\label{eq:proof_prop_int_rho_finite}\\
&=1+\sum_{n=1}^{\infty} \int_{\R} e^{r(u-S_n)} (J^{[n-1]}*J_{(\xi_n)})(du)\notag\\
&=1+\sum_{n=1}^{\infty} e^{-rS_n}e^{(n-1)\varphi(r)}Y_n\notag\\
&\leq B+ e^{\varphi(r)}\sum_{n=1}^{\infty} e^{(-r\varphi'(\lambda)+\varphi(r)+r\eps)n}Y_n,\notag
\end{align}
where $Y_n=\int_{\R}e^{ru}J_{(\xi_n)}(du)$ and $B$ is an a.s.\ finite $\AAA_{\xi}$-measurable random variable. Let us show that $\E [Y_n]$ is finite. Recall that the distribution of $\xi_n$ is $e^{\lambda z}J(dz)$. Using the properties of the Palm measures, we have
\begin{align*}
\E [Y_n]
&=
\int_{\R}\left(\int_{\R} e^{ru} J_{(z)}(du)\right)e^{\lambda z}J(dz)\\
&=\int_{\R}\E \left(\int_{\R}e^{ru+\lambda z}\chi_{(z)}(du)\right)J(dz)\\
&=\E \int_{\R}\int_{\R} e^{ru+\lambda z} \chi(du)\chi(dz)\\
&=e^{\varphi(r)+\varphi(\lambda)}.
\end{align*}
So, $\E[Y_n]$ is finite. 
If $\eps>0$ is chosen such that $\eps<\varphi'(\lambda)-\frac 1r \varphi(r)$, then the series on the right-hand side of~\eqref{eq:proof_prop_int_rho_finite} converges a.s. It follows that
$\E [\int_{\R}e^{ru}\rho(du)|\AAA_{\xi}]<\infty$ a.s.
Hence, $\int_{\R}e^{ru}\rho(du)<\infty$ a.s.
\end{proof}

We are now in position to prove Proposition~\ref{prop:family_drift}. Let $\sum_{i\in I} \delta_{P_i}$ be a realization of the point process $\Gamma$ from Proposition~\ref{prop:maillard}. Consider a branching dynamics $\{\gamma_n: n\in\N_0\}$  starting off $\Gamma$. That is, let $\{\chi_n^{(i)}:n\in\N_0\}$, $i\in  I$, be independent copies of the branching random walk $\{\chi_n:n\in\N_0\}$. Then, $\gamma_n=\sum_{i\in I} T_{P_i}(\chi_n^{(i)})$.  We need to show that
\begin{equation}\label{eq:proof_family_drift1}
\lim_{n\to\infty}\max \gamma_n=-\infty\;\;\;\text{a.s.}
\end{equation}
Let $\AAA_{\Gamma}$ be the $\sigma$-algebra generated by $\Gamma$. We have
$$
\P[P_i+\max\chi_n^{(i)}>-\eps n|\AAA_{\Gamma}]
\leq
e^{-n I(-\eps-\frac 1n P_i)}
\leq
e^{-n I(-\eps)} e^{I'(-\eps) P_i}.
$$
We have $I(-\eps)>0$ and $I'(-\eps)>r_0$ if $\eps>0$ is sufficiently small since $\lim_{\eps\to 0} I(-\eps)=I(0)=\inf_{t\in\R} \varphi(t)<0$ and $\lim_{\eps\to 0} I'(-\eps)=I'(0)=\arg\min \varphi>r_0$.
It follows from Lemma~\ref{lem:palm_exp} that
$$
\sum_{n=1}^{\infty}\P\left[\max \gamma_n>-\eps n|\AAA_{\Gamma}\right]< \infty\;\;\;\text{a.s.}
$$
An application of the Borel--Cantelli lemma completes the proof of~\eqref{eq:proof_family_drift1}.

\section{Proof of Theorem~\ref{theo:equilibrium_states} and Proposition~\ref{prop:k_property}}
The proof is based on the general results of~\cite{main_book}. By Theorem~2.3.6 of~\cite{main_book} a measure $\nu$ on $\R^d$ which is finite on bounded sets is an intensity of some equilibrium state  if and only if the branching dynamics starting off a Poisson point process with intensity $\nu$ converges weakly to a point process with intensity $\nu$ (such measures $\nu$ are called stable). Thus, by Theorem~\ref{theo:main}, the intensity $\nu$ of an equilibrium state $\Xi$ is a linear combination with non-negative coefficients of the measures of the form $e_{\lambda}$, where $\lambda\in K_{\text{st}}$.

By Theorem~4.7.28 in~\cite{main_book} any equilibrium state $\Xi$ with intensity $\nu$ is a mixture of the form~\eqref{eq:mixture} if the following two conditions (called aperiodicity and regularity) are satisfied:
\begin{enumerate}
\item  For every  probability measure $\sigma$ on $\R$ which is absolutely continuous with respect to the Lebesgue measure we have
\begin{equation}\label{eq:aperiodicity}
\lim_{n\to\infty}\|\sigma* D^{[n]}-\sigma*D^{[n+1]}\|=0,
\end{equation}
where $D$ is the probability measure defined in~\eqref{eq:def_D} and $\|\cdot\|$ denotes the total variation.
\item
For every bounded Borel set $A\subset \R$ and every $\eps>0$,
\begin{equation}\label{eq:regularity}
\lim_{n\to\infty} \nu\{z\in\R: J^{[n]}(A-z)>\eps\}=0.
\end{equation}
\end{enumerate}
See Proposition~3.8.4 in~\cite{main_book} for the first condition and Theorem~4.7.19 in~\cite{main_book} for the second one.
The aperiodicity condition holds as long as the measure $J$ is non-lattice by Theorem~11.10.4 and Proposition~11.9.1 of~\cite{matthes_etal_book}.
Let us verify the regularity condition.  In the subcritical case the measure on the left-hand side of~\eqref{eq:regularity} is $0$ for sufficiently large $n$. In the non-subcritical case, we need a different argument. Without restriction of generality let $\nu=e_{\lambda}$, where $\varphi(\lambda)=0$, $\lambda>0$, $\varphi'(\lambda)>0$. It suffices to consider the case $A=(a,\infty)$, where $a\in\R$.
We have, for every $t\geq 0$,
$$
J^{[n]}(A-z)\leq e^{-t (a-z)}\int_{\R} e^{t u}J^{[n]}(du)=e^{-t(a-z)+ n\varphi(t)}.
$$
It follows that
\begin{equation}\label{eq:hlf1}
\{z\in\R: J^{[n]}(A-z)>\eps\}\subset \left(-\frac{\varphi(t)}{t}n+\frac{\log \eps}{t}+a,\infty\right).
\end{equation}
We can choose $t>0$ such that $\varphi(t)<0$. For example, take $t=\lambda-\delta$ for sufficiently small $\delta>0$. The $e_{\lambda}$-measure of the set on the right-hand side of~\eqref{eq:hlf1} goes to $0$ as $n\to\infty$. This completes the proof of Theorem~\ref{theo:equilibrium_states}. Proposition~\ref{prop:k_property} follows from the regularity condition~\eqref{eq:regularity} by Theorem~4.7.19 of~\cite{main_book}.



\bibliographystyle{plainnat}
\bibliography{paper23bib}
\end{document}